\documentclass[12pt, reqno]{amsart}
\usepackage{mathrsfs, amsmath, amsthm, amscd, amsfonts, amssymb}
\usepackage{bm}

\newtheorem{theorem}{Theorem}[section]

\newtheorem{proposition}[theorem]{Proposition}

\numberwithin{equation}{section}

\newcommand{\vp}{\varphi}

\newcommand{\D}{\mathbb{D}}

\newcommand{\T}{\mathbb{T}}

\newcommand{\N}{\mathbb{N}}
\newcommand{\C}{\mathbb{C}}

\newcommand{\IN}{\mathbb{N}}

\begin{document}

\setcounter{page}{1}


\title[Fej\'er-Rogosinski theorem for the Neil Algebra]{Fej\'er-Rogosinski theorem for the Neil Algebra}

\author[Das]{Nilanjan Das}
\address{Indian Statistical Institute, Statistics and Mathematics Unit, 8th Mile, Mysore Road, Bangalore, 560059,
India}
\email{nilanjand7@gmail.com}

\author[Sarkar]{Jaydeb Sarkar}
\address{Indian Statistical Institute, Statistics and Mathematics Unit, 8th Mile, Mysore Road, Bangalore, 560059,
India}
\email{jay@isibang.ac.in, jaydeb@gmail.com}



\begin{abstract}
Partial sums of the Taylor expansions of functions from the Neil algebra are uniformly bounded by their uniform norms on the ball of radius $r$ centered at the origin, where $r$ denotes the unique real root of
\[
4x^3 + x - 2 = 0.
\]
Moreover, this radius $r$ is optimal.
\end{abstract}

\maketitle

\tableofcontents

\section{Introduction}\label{sec intro}

Let $\D$ denote the open unit disc in the complex plane $\C$, and let $\T=\partial\D$ denote its boundary, the unit circle. Given an analytic function $f$ on $\D$ (throughout this paper, all functions are assumed to be scalar-valued) with Taylor expansion
\[
f(z)=\sum_{n=0}^\infty a_n z^n,
\]
and a natural number $N \in \N$, we define the $N$-th partial sum $S_N(f)$ by
\[
(S_N(f))(z)=\sum_{n=0}^N a_n z^n.
\]
Denote by $H^\infty(\D)$ the Banach algebra of all bounded analytic functions on $\D$, equipped with the supremum norm:
\[
H^\infty(\D) = \{f \in \text{Hol}(\D): \|f\|_\infty:=\sup_{z\in\D}|f(z)|<\infty\}.
\]
There exist functions $f\in H^\infty(\D)$ for which the partial sums $(S_N(f))(z)$ can become arbitrarily large at certain points $z\in\T$ (cf. \cite{Fejer 2}). However, independently, Fej\'er and Rogosinski \cite{Fejer, Rogo} (also see Landau and Gaier \cite{LG}) showed that for all $f\in H^\infty(\D)$,
\begin{equation}\label{Intro-2}
\sup_{z\in\T} \left|(S_N(f))\left(\frac{1}{2}z\right)\right|\leq \|f\|_\infty,
\end{equation}
for all $N \in\N$. Moreover, the constant $\frac{1}{2}$ cannot be replaced by a larger number. In the context of Fej\'er-Rogosinski theorem on the general Hardy spaces and Banach spaces, see Aizenberg, Elin, and Shoikhet \cite{Aizenberg}, and Nabetani \cite{Nabetani}.

Our aim in this paper is to investigate the Fej\'er-Rogosinski theorem for a concrete closed subalgebra of $H^\infty(\D)$, namely, the Neil algebra, and to derive a sharper radius than $\frac{1}{2}$. The \textit{Neil algebra} $H^\infty_1(\D)$ is defined by
\[
H^\infty_1(\D):=\{f\in H^\infty(\D): f^\prime(0)=0\}.
\]
This is one of the simplest constrained algebras that serves as a test space for theories such as the interpolation problem, the corona theorem, the commutant lifting theorem, and many others in algebras beyond $H^\infty(\D)$ (cf. \cite{Scott, Davidson}). Yet, there is limited understanding regarding this function space. The main result of this paper is as follows:

\begin{theorem}\label{Neil}
Let $r_0$ denote the unique real root of the equation
\[
4r^3+r-2=0.
\]
Then, for all $f\in H^\infty_1(\D)$ and $N\in\IN$,
\[
\sup_{z\in\T}\left|(S_N(f))(r_0z)\right|\leq \|f\|_\infty.
\]
Moreover, the constant $r_0$ cannot be replaced by a larger number.
\end{theorem}

The next section, Section \ref{sec: Neil}, contains the proof of this result. A remark on $r_0$: Numerically (see the proof of Theorem \ref{thm: main} in Section \ref{sec: Neil}),
\[
r_0 \approx 0.6894.
\]
In particular,
\[
r_0 > \frac{1}{2}.
\]
This shows that seeking radii larger than $\frac{1}{2}$ for subalgebras of $H^\infty(\D)$ is a fruitful problem. To our knowledge, this is the first such attempt in the literature. On the other hand, in the third and final section, we show that this type of improvement need not hold in general for all subalgebras or subsets of $H^\infty(\D)$.

\section{On the Neil algebra}\label{sec: Neil}

The purpose of this section is to prove the main result of the paper. We begin by recalling another classical result, namely, a more general version of the Fej\'er-Rogosinski theorem due to Schur and Szeg\"o \cite{Schur-Szego}. For a given $N\in\IN$, denote by $r_N$ the largest number such that
\[
\frac{1}{2}+\sum_{n=1}^Nr^n\cos n\theta\geq 0,
\]
for all $r\leq r_N$. Let $N_0 \in \N$. The result of Schur and Szeg\"o \cite{Schur-Szego} states that $r_{N_0}$ is, in fact, the largest possible number such that
\[
\sup_{z\in\T}\left|(S_N(f))(r_{N_0}z)\right|\leq \|f\|_\infty,
\]
for all $N\geq N_0$ and all $f\in H^\infty(\D)$. Moreover, the sequence $\{r_N\}_{N=1}^\infty$ is monotonically increasing. Specifically,
\[
r_1= \frac{1}{2},
\]
and also,
\[
r_2=\sqrt{\frac{3}{8}} \approx 0.6123,\quad r_3\approx0.6478, \quad r_4\approx0.694,
\]
and so on.

The following coefficient estimate for power series, due to Wiener \cite[p. 4]{Bohr}, will be used frequently. Let $f \in H^\infty(\D)$ satisfy $\|f\|_\infty \leq 1$, and suppose that its Taylor series expansion is
\[
f(z) = \sum_{n=0}^{\infty} a_n z^n.
\]
A well-known method of Wiener \cite[p. 4]{Bohr} yields
\begin{equation}\label{eqn: a2 leq 1-a0}
|a_2| \leq 1 - |a_0|^2.
\end{equation}
To see this, define
\[
g(z) = \frac{f(z)+f(-z)}{2},
\]
for all $z \in \D$. Since $f:\mathbb{D}\to\mathbb{D}$ is analytic, it follows that $g$ is also an analytic self-map of $\D$. As
\[
g(z) =\sum_{n=0}^\infty a_{2n}z^{2n},
\]
for all $z \in \D$, replacing $z^2$ with $z\in\mathbb{D}$ shows that the function
\[
z \mapsto \sum_{n=0}^\infty a_{2n}z^n,
\]
is again an analytic self-mapping of $\mathbb{D}$. By the well-known Schwarz–Pick inequality (in its derivative form), $|a_2|\leq 1-|a_0|^2$, which proves the claim.

We now proceed to prove the main result, namely, Theorem \ref{Neil}. Building on the above, we introduce the following notation: for each $n \in \N$, let $R_n$ denote the largest number such that
\[
\sup_{z\in\T}\left|(S_n(f))\left(R_n z\right)\right|\leq\|f\|_\infty,
\]
for all $f\in H^\infty_1(\D)$. In the context of the main result, it is clear that we need to compute $\min_{n\in\IN} R_n$. In the following proof, we show that
\[
\min_{n\in\IN} R_n = R_3 = r_0.
\]
For the reader’s convenience, we restate it below. Recall that $r_0$ denotes the unique real root of the equation
\[
4r^3+r-2=0.
\]
Note that the existence and uniqueness of $r_0$ will also be established as part of the proof of the theorem below.

\begin{theorem}\label{thm: main}
For every $f\in H^\infty_1(\D)$ and $N\in\IN$,
\[
\sup_{z\in\T}\left|(S_N(f))(r_0z)\right|\leq \|f\|_\infty.
\]
Moreover,
\[
r_0 = R_3,
\]
and the constant $r_0$ cannot be replaced by a larger number.
\end{theorem}
\begin{proof}
The inequality is trivial for constant functions $f$ in $H^\infty_1(\D)$. Therefore, in what follows, we always assume that the functions $f$ considered in $H^\infty_1(\D)$ are nonconstant. We also note that it suffices to prove the stated inequality for functions satisfying $\|f\|_\infty = 1$, since the general case then follows by a simple scaling argument. Thus, without loss of generality, we assume throughout that
\[
\|f\|_\infty=1.
\]
For an arbitrary $f \in H^\infty_1(\D)$ with $\|f\|_\infty = 1$, write its power series representation on $\D$ as
\[
f(z) = \sum_{n=0}^{\infty} a_n z^n.
\]
We know that
\[
a_1 = f'(0) = 0.
\]
This implies
\[
R_1=1.
\]
Next, we note that
\[
|a_0| < 1.
\]
Indeed, if $|a_0| = |f(0)| = 1$, then by the maximum modulus theorem, $f$ must be a constant function. To compute $R_2$, we recall from \eqref{eqn: a2 leq 1-a0} that $|a_2| \leq 1 - |a_0|^2$, and hence, for $z\in\T$, we have
\[
\begin{split}
\left|(S_2(f)) \left(\frac{z}{\sqrt{2}}\right)\right| & = \left|\sum_{n=0}^2 a_n\left(\frac{z}{\sqrt{2}}\right)^n\right|
\\
& \leq |a_0| + \frac{1}{2} |a_2|
\\
& \leq |a_0| + \frac{1}{2} (1-|a_0|^2)
\\
&
= 1 - \frac{1}{2} (1 - |a_0|)^2
\\
& \leq 1,
\end{split}
\]
which yields
\[
R_2 \geq \frac{1}{\sqrt{2}}.
\]
We claim that $R_2 = \frac{1}{\sqrt{2}}$. To this end, for each $\alpha\in(0,1)$, we consider the function $b_\alpha \in H^\infty_1(\D)$ defined by
\[
{b}_\alpha(z)=\frac{z^2+\alpha}{1+\alpha z^2}.
\]
Note that $(S_2(b_\alpha))(z) = \alpha + (1 - \alpha^2) z^2$ and $\|b_\alpha\|_\infty = 1$. Therefore, for each $r \in (0,1)$, we have
\[
|(S_2(b_\alpha)) (r)| = \alpha + (1-\alpha^2)r^2,
\]
and hence
\[
|(S_2(b_\alpha)) (r)| \leq 1,
\]
if and only if (note that $\alpha < 1$)
\[
r\leq \frac{1}{\sqrt{1+\alpha}}.
\]
Since
\[
\frac{1}{\sqrt{1+\alpha}}\to\frac{1}{\sqrt{2}},
\]
as $\alpha\to 1-$, it follows that
\[
R_2=\frac{1}{\sqrt{2}}.
\]
As the next step, we compute $R_3$. Consider the analytic self-map $\vp$ of $\D$ defined by (recall that $f'(0) = 0$)
\[
\vp(z)=\frac{f(z)-a_0}{z^2(1-\overline{a_0}f(z))}.
\]
Write the Taylor series expansion of $\vp$ on $\D$ as
\[
\vp(z)=\sum_{n=0}^\infty \vp_nz^n.
\]
Since
\[
f(z) = a_0 + z^2\vp(z)(1-\overline{a_0}f(z)),
\]
it follows that
\[
(S_3(f))(z)=a_0+z^2(1-|a_0|^2)((S_1(\vp))(z)).
\]
By \eqref{eqn: a2 leq 1-a0}, we know that $|\vp_1|\leq1-|\vp_0|^2$. This gives, for $r>0$ and $z \in \T$,
\[
|(S_1(\vp))(rz)| = |\vp_0 + \vp_1 rz| \leq |\vp_0| + (1 - |\vp_0|^2)r,
\]
and hence
\[
|(S_3(f))(rz)| \leq |a_0| + r^2(1-|a_0|^2)(|\vp_0|+(1-|\vp_0|^2)r).
\]
Observe that
\[
|a_0|+r^2(1-|a_0|^2)(|\vp_0|+(1-|\vp_0|^2)r)\leq 1,
\]
is equivalent to the condition that
\[
r^2(1-|a_0|^2)(|\vp_0|+(1-|\vp_0|^2)r)\leq 1 - |a_0|.
\]
Since $(1-|a_0|^2) = (1+|a_0|)(1-|a_0|)$ and $|a_0| < 1$, the above is equivalent to
\[
r^2(1 + |a_0|)(|\vp_0|+(1-|\vp_0|^2)r)\leq 1.
\]
Finally, since $1 + |a_0| < 2$, we have
\[
r^2(1 + |a_0|)(|\vp_0|+(1-|\vp_0|^2)r) \leq 2 r^2 (|\vp_0|+(1-|\vp_0|^2)r).
\]
In particular,
\[
|(S_3(f))(rz)| \leq 1,
\]
whenever
\[
2 r^2 (|\vp_0|+(1-|\vp_0|^2)r) \leq 1.
\]
To estimate the quantity $|\vp_0|+(1-|\vp_0|^2)r$, we pick $r\in (\frac{1}{2}, 1)$, and define $\psi_r(x) :(0,1)\to(0, \infty)$ by
\[
\psi_r(x)=x+(1-x^2)r.
\]
Since $\psi_r^\prime(x)=1-2xr$, it follows that
\[
\psi_r^\prime(x) = 0,
\]
whenever $x = \frac{1}{2r}\in (\frac{1}{2}, 1)$. Moreover, $\psi$ is increasing on $(\frac{1}{2}, \frac{1}{2r})$ and decreasing on $(\frac{1}{2r}, 1)$. Therefore,
\[
\sup_{x\in (0,1)} \psi_r(x) = \psi_r \left(\frac{1}{2r}\right) = r + \frac{1}{4r},
\]
and hence
\[
2r^2(|\vp_0|+(1-|\vp_0|^2)r)\leq 2r^3+\frac{r}{2}.
\]
Thus, we conclude that $|(S_3(f))(rz)|\leq1$ for $r\in (\frac{1}{2},1)$ and $z \in \T$ whenever $2r^3 + \frac{r}{2} \leq 1$, that is,
\[
4r^3+r\leq 2.
\]
This inequality is satisfied for $r\leq r_0$, where $r_0$ is the unique real root of
\[
p(r):=4r^3+r-2=0,
\]
in the interval $(\frac{1}{2}, 1)$. Note that
\[
p\Big(\frac{1}{2}\Big)<0 \text{ and } p(1)>0,
\]
and
\[
p^\prime(r)>0,
\]
for all $r\in \mathbb{R}$, which guarantees the existence and uniqueness of $r_0$. Now, from the Fej\'er-Rogosinski theorem \cite{Fejer, Rogo}, we already know that $R_3\geq \frac{1}{2}$, and hence $|(S_3(f))(rz)|\leq1$ is valid for all $r\leq r_0$, that is,
\[
R_3\geq r_0.
\]
To show that equality holds, we pick $\alpha\in (0,1)$ and consider the function ${b}_{\alpha, r_0}$ on $\D$, where
\[
{b}_{\alpha, r_0}(z) = \frac{z^2(2r_0z+1)+\alpha(2r_0+z)}{\alpha z^2(2r_0z+1)+2r_0+z}.
\]
Given that $\frac{1}{2r_0} \in (\frac{1}{2},1)$, define
\[
b(z) = \frac{z + \frac{1}{2r_0}}{1 + \frac{z}{2 r_0}},
\]
for all $z \in \D$. Then
\[
b_{\alpha, r_0}(z) = \frac{z^2b(z) + \alpha}{1 + \alpha z^2b(z)}.
\]
This, in particular, shows that ${b}_{\alpha, r_0}(z)$ is an analytic self-mapping of $\D$. Note that $b_{\alpha, r_0}(1) = 1$, which implies $\|b_{\alpha, r_0}\|_\infty = 1$. Moreover, we have
\[
{b}_{\alpha, r_0}(z) = \alpha + (1 - \alpha^2) z^2 b(z) - \alpha (1 - \alpha^2) z^4 b^2(z) + \cdots.
\]
On the other hand, since
\[
\alpha + (1 - \alpha^2) z^2 b(z) = \alpha + (1 - \alpha^2) z^2 \Big(\frac{1}{2r_0} + \Big(1 - \frac{1}{4r_0^2}\Big) z - \cdots \big),
\]
it follows that
\[
(S_3(b_{\alpha, r_0}))(r)=\alpha+(1-\alpha^2)r^2\left(\frac{1}{2r_0}+\left(1-\frac{1}{4r_0^2}\right)r\right),
\]
for any $r>0$. Moreover, for any $\epsilon>0$, we get
\[
\begin{split}
\widetilde{r_0} : & = (r_0+\epsilon)^2\left(\frac{1}{2r_0}+\left(1-\frac{1}{4r_0^2}\right)(r_0+\epsilon)\right)
\\
& > r_0^2\left(\frac{1}{2r_0}+\left(1-\frac{1}{4r_0^2}\right)r_0\right)
\\
& = \frac{p(r_0)+2}{4}
\\
& =1/2,
\end{split}
\]
that is, $\frac{1}{\widetilde{r_0}}-1<1$. Let us now choose any
\[
\alpha\in\left(\max\left\{0, \frac{1}{\tilde{r_0}}-1\right\},1\right).
\]
For the above $\alpha$, we have
\[
\begin{split}
(S_3(b_{\alpha, r_0}))(r_0+\epsilon) = \alpha+(1-\alpha^2)\tilde{r_0} = \alpha+(1-\alpha)(1+\alpha)\tilde{r_0} >\alpha+(1-\alpha) =1.
\end{split}
\]
As a result, we conclude that
\[
R_3=r_0.
\]
Note now that $p(1/\sqrt{2})>0$, and therefore,
\[
R_2> r_0=R_3\approx 0.6894.
\]
We now show that $R_N>R_3$ for all $N\geq 4$. To this end, we will use a result of Schur and Szeg\"{o} \cite{Schur-Szego} (see the beginning of this section). Since $H^\infty_1(\D)\subsetneq H^\infty(\D)$, using the monotonicity of $r_N$'s, we have
\[
R_N\geq r_N\geq r_4,
\]
for all $N\geq4$. Using the approximate value
\[
r_4\approx 0.694,
\]
as stated in \cite{Schur-Szego}, we immediately conclude that $r_4>R_3 \approx 0.6894$. This now asserts that
$$
\min_{N\in\IN}R_N=\min\{R_1, R_2, R_3, r_4\}=R_3=r_0,
$$
and our proof is therefore complete.
\end{proof}

\section{On $H^\infty_0(\D)$}

One of the purposes of this paper is to sharpen the Fej\'er–Rogosinski theorem in the context of subalgebras and subsets of $H^\infty(\D)$. The result for the Neil algebra shows that such a sharpening is indeed possible. In this section, we show that such an improvement does not necessarily hold in general. We investigate this question for the subalgebra $H^\infty_0(\D)$ of $H^\infty(\D)$, where
\[
H^\infty_0(\D):=\{f\in H^\infty(\D): f(z)\neq 0\,\,\mbox{for all}\,\,z\in\D\}.
\]
This subalgebra contains, in particular, the family of singular inner functions, a class of fundamental importance in analytic function theory and the theory of Hilbert function spaces.

\begin{proposition}\label{Nonzero}
For all $f\in H^\infty_0(\D)$, the inequality \eqref{Intro-2} remains valid for every $N\in\IN$, and this quantity $\frac{1}{2}$ is the best possible.
\end{proposition}
\begin{proof}
Since $H^\infty_0(\D)\subsetneq H^\infty(\D)$, the validity of \eqref{Intro-2} is automatically guaranteed, and hence we only need to establish the optimality of $1/2$ for the functions in $H^\infty_0(\D)$. Pick $\alpha\in (0,1)$, and consider the singular inner function $f_\alpha \in H^\infty_0(\D)$ defined by
\[
f_\alpha(z)=\exp\left(\alpha\frac{z+1}{z-1}\right).
\]
It is known that $\|f_\alpha\|_\infty=1$. A little calculation shows that for any $r>0$,
\[
\left|(S_1(f_\alpha))(-r)\right| = e^{-\alpha}\left(1+2\alpha r\right).
\]
Therefore,
\[
\left|(S_1(f_\alpha))(-r)\right| \leq 1,
\]
if and only if
\[
r\leq \frac{e^\alpha-1}{2\alpha}.
\]
It is easy to see that
\[
\frac{e^\alpha-1}{2\alpha} > \frac{1}{2},
\]
for all $\alpha>0$, and
\[
\lim_{\alpha\to0+}\frac{e^\alpha-1}{2\alpha}=\frac{1}{2}.
\]
Clearly, for any $\epsilon>0$, there is $\alpha_0\in (0,1)$ such that
\[
\frac{1}{2}< \frac{e^{\alpha_0}-1}{2\alpha_0}<\frac{1}{2}+\epsilon.
\]
Consequently,
\[
\left|S_1\left(f_{\alpha_0}\right)\left(-\frac{1}{2}-\epsilon\right)\right| > \left|S_1\left(f_{\alpha_0}\right)\left(-\frac{e^{\alpha_0}-1}{2\alpha_0}\right)\right| = 1.
\]
This shows that the constant $\frac{1}{2}$ in the Fej\'er-Rogosinski theorem cannot be replaced by any larger number.
\end{proof}

The above result, together with Theorem \ref{Neil}, shows that the classical Fej\'er–Rogosinski theorem cannot be further improved for $H^\infty_0(\D)$, and in contrast, for $H^\infty_1(\D)$, the constant $\frac{1}{2}$ can be replaced by a larger number. A closer examination of the proof of the above proposition shows that, even when restricted to the family of singular inner functions, the constant $\frac{1}{2}$ in the Fej\'er–Rogosinski theorem cannot be improved.

\vspace{0.1in}

\noindent\textbf{Acknowledgement:}
The first named author is supported by the National Postdoctoral Fellowship (N-PDF) provided by the Anusandhan National Research Foundation, India (File number: PDF/2025/000089). The research of the second named author is supported in part by ANRF, Department of Science
\& Technology (DST), Government of India (File No: ANRF/ARGM/2025/000130/MTR).

\end{document}